\documentclass[10pt,a4paper]{article}

\usepackage[final]{optional}

\usepackage{amsfonts, amsmath, wasysym}

\usepackage[pdftex]{graphicx}

\usepackage{amssymb,amsthm,
paralist
}

\usepackage{
latexsym,
}

\usepackage[usenames]{color}

\usepackage{url}

\definecolor{darkgreen}{rgb}{0,0.5,0}
\definecolor{darkred}{rgb}{0.7,0,0}
\usepackage[colorlinks, 
citecolor=darkgreen, linkcolor=darkred
]{hyperref}

\theoremstyle{plain}
\newtheorem{lemma}{Lemma}[section]
\newtheorem{thm}[lemma]{Theorem}

\newtheorem{cor}[lemma]{Corollary}

\theoremstyle{definition}

\newtheorem{rmk}[lemma]{Remark}

\numberwithin{equation}{section}

\newcommand{\m}{\ensuremath{{\cal M}}}

\newcommand{\cc}{\ensuremath{{\cal C}}}

\newcommand{\cl}{\ensuremath{{\cal L}}}

\newcommand{\pl}[2]{{\frac{\partial #1}{\partial #2}}}

\newcommand{\al}{\alpha}
\newcommand{\be}{\beta}
\newcommand{\ga}{\gamma}

\newcommand{\de}{\delta}

\newcommand{\Om}{\Omega}

\newcommand{\la}{\lambda}

\renewcommand{\th}{\theta}

\newcommand{\ep}{\varepsilon}

\newcommand{\R}{\ensuremath{{\mathbb R}}}
\newcommand{\N}{\ensuremath{{\mathbb N}}}

\newcommand{\Z}{\ensuremath{{\mathbb Z}}}

\newcommand{\downto}{\downarrow}

\newcommand{\lap}{\Delta}

\DeclareMathOperator{\interior}{int}
\DeclareMathOperator{\inj}{inj}

\newcommand{\beq}{\begin{equation}}
\newcommand{\eeq}{\end{equation}}
\newcommand{\beqa}{\begin{equation}\begin{aligned}}
\newcommand{\eeqa}{\end{aligned}\end{equation}}
\newcommand{\brmk}{\begin{rmk}}
\newcommand{\ermk}{\end{rmk}}
\newcommand{\partref}[1]{\hbox{(\csname @roman\endcsname{\ref{#1}})}}
\newcommand{\half}{\frac{1}{2}}

\newcommand{\Rm}{{\mathrm{Rm}}}
\newcommand{\Ric}{{\mathrm{Ric}}}

\newcommand{\Do}{\ensuremath{{D\backslash\{0\}}}}

\newcommand*\dz{\ensuremath{\mathrm{d}z}}
\newcommand*\pddt{\ensuremath{\frac{\partial}{\partial t}}}
\newcommand*\ee{\ensuremath{\mathop{\mathrm{e}}\nolimits}}
\usepackage{soul}

\title{{
\bf
Remarks on Hamilton's Compactness Theorem for Ricci 
flow 
}
\\ 
}
\author{Peter M. Topping}
\date{\today}

\begin{document}

\maketitle
\parskip=10pt

\begin{abstract}
A fundamental tool in the analysis of Ricci flow is a compactness result of Hamilton in the spirit of the work of Cheeger, Gromov and others.
Roughly speaking it allows one to take a sequence of Ricci flows with uniformly bounded curvature and uniformly controlled injectivity radius, and extract a subsequence that converges to a complete limiting Ricci flow. A widely quoted extension of this result allows the curvature to be bounded  uniformly only in a local sense. However, in this note we give a counterexample.
\end{abstract}

\section{Introduction}
\label{intro}

\subsection{Hamilton-Cheeger-Gromov convergence and compactness}
\label{hcg_sect}



Let $(\m_i,g_i(t))$ be a sequence of smooth families of complete 
Riemannian manifolds for $t \in (a,b)$ where $-\infty \leq a<0<b \leq \infty$, and suppose $p_i \in \m_i$ for each $i$. 
Hamilton \cite{hamcptness} 
introduced a notion of convergence of such `flows' to another smooth family $(\m,g(t))$ of Riemannian manifolds for 
$t \in (a,b)$ based on the notion of Cheeger-Gromov convergence of Riemannian manifolds. As discussed in \cite[Chapter 7]{RFnotes}, we must take points $p_i\in\m_i$ and $p \in \m$, and can then say that
\begin{equation*}
\begin{aligned}
(\m_i,g_i(t),p_i) \to (\m,g(t),p)
\end{aligned}
\end{equation*}
as $i\to\infty$ in the sense of Hamilton-Cheeger-Gromov if there exist 
\begin{compactenum}[(a)] \item a sequence of compact $\Omega_i \subset \m$
exhausting $\m$ with $p \in \interior (\Omega_i)$ for each $i$, and
\item a sequence of smooth maps $\phi_i : \Omega_i \to \m_i$, 
diffeomorphic onto their image, and with $\phi_i(p)=p_i$,
\end{compactenum} such that
\begin{equation*}
\begin{aligned}
\phi^*_i g_i(t) \to g(t)
\end{aligned}
\end{equation*}
smoothly locally on $\m \times (a,b)$
as $i\to\infty$.
%


If the limit $g(t)$ is complete at some time $t_0\in (a,b)$, then this is the unique limit that is complete at time $t_0$, modulo a time-independent isometry. 

In the modern theory of Ricci flow, one of the most fundamental tools is the following compactness theorem of Hamilton
\cite{hamcptness} which holds if one has suitable control on the curvature -- hypthesis (i) -- and the injectivity radius -- hypothesis (ii), where $\inj(\m,g,p)$ denotes the injectivity radius of $(\m,g)$ at $p\in\m$.
For further information and intuition, see \cite[Chapter 7]{RFnotes} and \cite{formations}.

\begin{thm}[Compactness of Ricci flows; Hamilton \cite{hamcptness}]\label{THM_RFcompactness}
Suppose that  $(\m_i,g_i(t))$ is a sequence of complete Ricci flows on $n$-dimensional manifolds $\m_i$
for $t \in (a,b)$, where $-\infty \leq a<0<b \leq \infty$. 
Suppose that $p_i \in \m_i$ for each $i$, and that 
\begin{equation*}
\begin{aligned}
(i)\qquad \qquad & \sup_i \sup_{x \in \m_i,~t \in (a,b)} 
\big|\Rm(g_i(t))\big|(x) < \infty;\qquad\text{and} \\
(ii)\qquad\qquad &\inf_i~ \inj (\m_i,g_i(0),p_i) > 0.
\end{aligned}
\end{equation*}
Then there exist a manifold $\m$ of dimension $n$, 
a complete Ricci flow 
$g(t)$ on \m\ for $t \in (a,b)$, and a point $p \in
\m$ such that, after passing to a subsequence in $i$, we have
\begin{equation*}
\begin{aligned}
(\m_i,g_i(t),p_i) \to (\m,g(t),p),
\end{aligned}
\end{equation*}
as $i\to\infty$.
\end{thm}

In many important applications, particularly while making contradiction arguments, the uniform upper bound for curvature is too strong a hypothesis. One can weaken this by assuming only local bounds on the curvature, without significant alteration of the proof. For example, we have:

\begin{thm}[Compactness of Ricci flows: Extension 1]
\label{ct2}
Suppose that  $(\m_i,g_i(t))$ is a sequence of complete Ricci flows on $n$-dimensional manifolds $\m_i$
for $t \in (a,b)$, where $-\infty \leq a<0<b \leq \infty$. 
Suppose that $p_i \in \m_i$ for each $i$, and that
\begin{compactenum}[(i)]
\item
for all $r>0$, there exists $M=M(r)<\infty$ such that for all $t\in (a,b)$ and for all $i$, there holds
$$\sup_{B_{g_i(0)(p_i,r)}} 
\big|\Rm(g_i(t))\big|_{g_i(t)} \leq M,\qquad\text{and}$$
\item
$$\inf_i~ \inj (\m_i,g_i(0),p_i) > 0.$$
\end{compactenum}
Then there exist a manifold $\m$ of dimension $n$, 
a Ricci flow 
$g(t)$ on \m\ for $t \in (a,b)$, and a point $p \in
\m$ such that $(\m,g(0))$ is complete, and after passing to a subsequence in $i$, we have
\begin{equation}
\label{cgnce_statement}
\begin{aligned}
(\m_i,g_i(t),p_i) \to (\m,g(t),p),
\end{aligned}
\end{equation}
as $i\to\infty$.
\end{thm}

In the literature this result is often stated with the additional conclusion that the limit Ricci flow $g(t)$ is \emph{complete} -- that is, $g(t)$ is complete for every $t\in (a,b)$ -- which is often important in applications (see also Section \ref{trueorfalse}). It has been noted very recently \cite[Appendix E]{KLv4} following objections by Cabezas-Rivas and Wilking that this completeness `does not immediately follow'. In this note we demonstrate that in fact the completeness fails in general.

\subsection{Contracting cusps - intuition behind the construction}

In this section we sketch an intuitive construction which indicates that a limiting Ricci flow arising in Theorem \ref{ct2} can be complete over an open time interval containing $t=0$, but incomplete beyond a certain time. The precise construction we make in Section \ref{precisesection} will be a little different to reduce the amount of technology required to make it rigorous.

At the core of the construction are the `contracting cusp' examples of Ricci flows we constructed in \cite{revcusp}. In that work we considered starting a Ricci flow with a complete, bounded-curvature, noncompact surface whose ends were asymptotic to hyperbolic cusps, in some weak sense. There is a standard way of flowing such a surface (constructed in \cite{GT2}, extending the flow of Hamilton and Shi) during which the cusps remain in place. A simple example which suffices for now would be to consider a complete hyperbolic surface $(T^2_p,g_{hyp})$ that is conformally  a torus $T^2$ punctured at one point $p\in T^2$, and take the Ricci flow which evolves simply by dilation:
$$g_1(t):=(1+2t)g_{hyp},$$
for all $t\in [0,\infty)$.
However, the central message of \cite{revcusp} is that there is another way of flowing such a surface where one imagines capping off the hyperbolic cusp infinitely far out, and letting it contract down. More precisely, and most concretely, we showed that it is possible to find a smooth Ricci flow $g_2(t)$ on the torus $T^2$, for $t\in (0,\infty)$ so that $g_2(t)\to g_{hyp}$ smoothly locally on $T^2_p:=T^2\backslash \{p\}$ as $t\downto 0$. One can view a cusp as developing at $p$ as we go backwards in time to $t=0$ -- a so-called `reverse singularity' (see also \cite{rick_survey}). The flow that contracts the cusp is unique as described in \cite[Theorem 1.5]{revcusp}. 

One way of viewing this now would be as a way of constructing an example of \emph{nonuniqueness} for Ricci flow. In particular, a perfectly valid smooth Ricci flow on $T^2_p$ would consist of $g_1(t+1)$ for $t\in [-1,1]$, followed by an appropriate scaling of $g_2(t)$, restricted to $T^2_p$. Precisely, that scaled flow would be the restriction to $T^2_p$ of $5g_2((t-1)/5)$ for $t\in (1,\infty)$. Beyond $t=1$, the flow would be incomplete because we have removed the point $p$.

The core principle of this paper is:
\begin{quote}
\em
Such flows can arise as Hamilton-Cheeger-Gromov limits of complete Ricci flows within the extension of Hamilton's Compactness Theorem given in Theorem \ref{ct2}.
\end{quote}

A precise statement will be given in Section \ref{precisesection}, but first we sketch how one could hope to construct a sequence of complete, bounded-curvature Ricci flows satisfying the hypotheses of Theorem \ref{ct2}, with a limit flow as given above.

The basic building blocks are the complete hyperbolic metric $g_{hyp}$ on $T^2_p$ and a complete metric $g_{bulb}$ on the punctured 2-sphere (equivalently on the plane) whose end is asymptotic to a hyperbolic cusp, and whose area is exactly $8\pi$.  There are many results which will give a complete Ricci flow starting at $g_{bulb}$, and describe its properties -- see \cite{ham_dask}, \cite{dask_sesum}, \cite{GT2} and the references therein. The important point to note is that it will flow for exactly time $2$ before the area decreases to zero. During this time, area gets sucked out of the bulb part of the manifold, and at each time the flow has a cusp-like end. 

The Ricci flows $g_k(t)$ we wish to imagine putting into Theorem \ref{ct2} will exist on the whole of $T^2$. They will be the Ricci flows whose initial data at $t=-1$ arises by chopping off the ends of the cusps of both the metrics $g_{hyp}$ and $g_{bulb}$, and gluing them together. The larger $k$ becomes, the further out we wish to make our truncations. 
If we fix a base-point $q\in T^2\backslash\{p\}$, and consider $p$ to correspond to some fixed point in the bulb surface for each $k$, then the distance $d_{g_k(-1)}(p,q)$ converges to infinity as $k\to\infty$.

The idea then is that the two distinct parts of $g_k(-1)$ will evolve largely independently of each other within the time interval $[-1,1)$,
but then at time $t=1$, the area within the bulb part of the flow will have been exhausted, and the cusp end of the torus part of the flow should start contracting. The effect of this is that during the initial time interval $[-1,1)$, the Ricci flows $(T^2,g_k(t),q)$ will converge to a homothetically expanding  hyperbolic metric on $T^2_p$ (with the bulb part too far away to see) but for $t>1$, the manifolds 
$(T^2, g_k(t),q)$ should have uniformly controlled diameter (independent of $k$) and will converge to a smooth metric on the whole torus in the Cheeger-Gromov sense.
The construction is illustrated in the following figure.

\def\svgwidth{400pt}
\begingroup%
  \makeatletter%
  \providecommand\color[2][]{%
    \errmessage{(Inkscape) Color is used for the text in Inkscape, but the package 'color.sty' is not loaded}%
    \renewcommand\color[2][]{}%
  }%
  \providecommand\transparent[1]{%
    \errmessage{(Inkscape) Transparency is used (non-zero) for the text in Inkscape, but the package 'transparent.sty' is not loaded}%
    \renewcommand\transparent[1]{}%
  }%
  \providecommand\rotatebox[2]{#2}%
  \ifx\svgwidth\undefined%
    \setlength{\unitlength}{841.88974609bp}%
    \ifx\svgscale\undefined%
      \relax%
    \else%
      \setlength{\unitlength}{\unitlength * \real{\svgscale}}%
    \fi%
  \else%
    \setlength{\unitlength}{\svgwidth}%
  \fi%
  \global\let\svgwidth\undefined%
  \global\let\svgscale\undefined%
  \makeatother%
  \begin{picture}(1,0.70707072)%
    \put(0,0){\includegraphics[width=\unitlength]{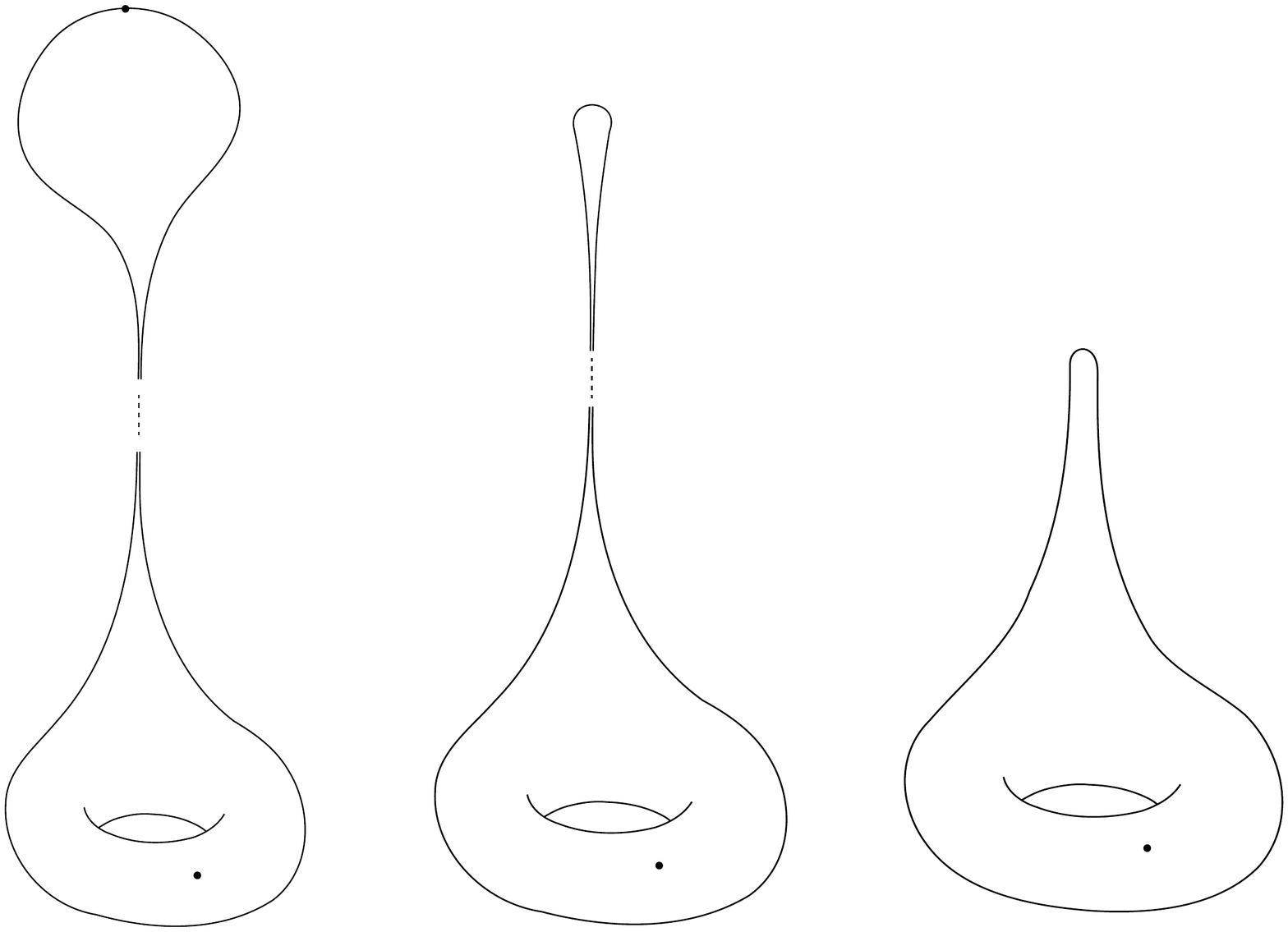}}%
    \put(0.04851812,0.55440036){\color[rgb]{0,0,0}\makebox(0,0)[lb]{\smash{$g_{bulb}$:}}}%
    \put(0.05110575,0.18695646){\color[rgb]{0,0,0}\makebox(0,0)[lb]{\smash{$g_{hyp}$:}}}%
    \put(0.20442301,0.05369338){\color[rgb]{0,0,0}\makebox(0,0)[lb]{\smash{$t=-1$}}}%
    \put(0.4774183,0.05692794){\color[rgb]{0,0,0}\makebox(0,0)[lb]{\smash{$t \sim 1$}}}%
    \put(0.77240846,0.06145627){\color[rgb]{0,0,0}\makebox(0,0)[lb]{\smash{$t>1$}}}%
    \put(0.27234838,0.13520382){\color[rgb]{0,0,0}\makebox(0,0)[lb]{\smash{$q$}}}%
    \put(0.54469669,0.14037906){\color[rgb]{0,0,0}\makebox(0,0)[lb]{\smash{$q$}}}%
    \put(0.83774621,0.1507296){\color[rgb]{0,0,0}\makebox(0,0)[lb]{\smash{$q$}}}%
    \put(0.21412664,0.66372785){\color[rgb]{0,0,0}\makebox(0,0)[lb]{\smash{$p$}}}%
  \end{picture}%
\endgroup%

It is not too difficult to argue precisely, using pseudolocality technology (see \cite{P1}) that one obtains the expanding hyperbolic metric flow as limit on an initial time interval. In order to show that the cusp will start collapsing at a uniformly controlled rate  before some uniformly bounded time we require some more involved \emph{a priori} estimates. To simplify this aspect, we will make a slight adjustment of the construction outlined above, notably replacing the bulb metric by long thin cigars with $k$-dependent geometry, but with area uniformly bounded above and below. We will then be able to appeal to some of the ideas from \cite{revcusp}. Precise assertions are given in the next section.

\subsection{A precise example}
\label{precisesection}

We will work on the two-dimensional unit disc $D$ in the plane, centred at the origin, and the same disc punctured at the origin. All metrics considered will respect the conformal structure inherited from the plane. (Ricci flow in two-dimensions preserves the conformal class -- see Appendix \ref{2Dappendix}.)
From now on in this paper, $g_{hyp}$ will always refer to the unique complete hyperbolic metric on $\Do$, which can be written $e^{2h}|\dz|^2:=e^{2h}(dx^2+dy^2)$ where $h=-\ln[r(-\ln r)]$.

\begin{thm}[Main Theorem]
\label{mainthm}
There exist a sequence of smooth, complete Ricci flows 
$(D,g_k(t))$ and a smooth Ricci flow 
$(\Do,g_\infty(t))$, all defined for $t\in [0,\infty)$, such that
\begin{enumerate}
\item
$g_k(t)\to g_\infty(t)$ smoothly locally on 
$(\Do)\times [0,\infty)$ as $k\to\infty$;
\item
for $t\in [0,\eta]$ (for some $\eta\in (0,1)$) we have $g_\infty(t)=(1+2t)g_{hyp}$;
\item
for $t\geq 1$, the manifold $(\Do,g_\infty(t))$ is \emph{not} complete;
\item
for $t\geq 1$, the limit $g_\infty(t)$ can be extended across the origin to give a smooth Ricci flow on the whole of $D$, and $g_k(t)$ converges smoothly locally to that extension on 
$D\times [1,\infty)$ as $k\to\infty$;
\item
for all $\ep\in (0,\half)$, there exists $C<\infty$ (independent of $k$) such that the Gauss curvature $K$ is controlled by
$$\sup_{(D_{1-\ep}\backslash D_\ep)\times [0,\infty)}\big|K[g_k(t)]\big|_{g_k(t)}<C\qquad\text{ and }\qquad
\sup_{D_{1-\ep}\times [1,\infty)}\big|K[g_k(t)]\big|_{g_k(t)}<C.$$
\end{enumerate}
\end{thm}
One consequence of this theorem is that if we take any $q\in \Do$
and define $\tilde g_k(t)=g_k(t+\eta/2)$ and 
$\tilde g_\infty(t)=g_\infty(t+\eta/2)$, and consider these flows on the time interval $(-\eta/2, 1)$, say, then
we have that for all $r>0$ there exists $C<\infty$ such that
$$\sup_{B_{\tilde g_k(0)}(q,r)\times (-\eta/2,1)}
\big|K[\tilde g_k(t)]\big|_{\tilde g_k(t)}<C,$$
and we may apply Theorem \ref{ct2} to $(D,\tilde g_k(t),q)$. 
By the uniqueness mentioned in Section \ref{hcg_sect}, the limit must be 
$(\Do,\tilde g_\infty(t),q)$ (up to a time-independent isometry) and for $t\geq 1-\eta/2$, the Riemannian manifold $(\Do,\tilde g_\infty(t))$ is not complete, so we deduce:

\begin{cor}
There exists a sequence of Ricci flows to which Theorem \ref{ct2}
applies, but for which the resulting limit is not complete at certain times.
\end{cor}

Note that if we fix $t\in [1-\eta/2,1)$, then the sequence of pointed complete Riemannian manifolds $(D,\tilde g_k(t),q)$ will converge in the Cheeger-Gromov sense to the pointed complete Riemannian manifold 
$(D,\tilde g_\infty(t),q)$ where we use the same notation $g_\infty(t)$ and $\tilde g_\infty(t)$ to refer to the extension across the origin from Part 4 of the theorem.
(For a discussion of Cheeger-Gromov convergence, see \cite[Chapter 7]{RFnotes}, for example.)

\begin{rmk}
The techniques introduced in this paper can be adapted to show that on any Riemann surface, there exists an immortal Ricci flow $g(t)$ for $t\in [0,\infty)$ that has bounded curvature for $t\in [0,1)$, unbounded curvature for $t\in [1,T]$ (for some $T>1$) but then bounded curvature again for sufficiently large $t$.
In fact, in principle one should be able to prescribe the subset ${\cal I}\subset (0,\infty)$ of time on which the curvature of the Ricci flow should be unbounded, with a high degree of generality.
One other special case is that it is possible to find an immortal Ricci flow on any Riemann surface which has bounded curvature for $t\in [0,1)$ but unbounded curvature for $t\geq 1$. Examples of Ricci flows with unbounded curvature in three dimensions, and with 
unbounded curvature for $t\geq 1$ in four dimensions, were constructed by Cabezas-Rivas and Wilking \cite{CWv4}. 
Ricci flows on surfaces with unbounded curvature for all $t\geq 0$
were first constructed in \cite{GT3}. Details of these new constructions will appear in \cite{GTinprep}.
\end{rmk}

This paper is organised as follows. In Section \ref{trueorfalse} we discuss various extensions of Hamilton's compactness theorem that can be used in applications in place of the various incorrect assertions that exist in the literature. In Section 
\ref{main_construct_sect}, we make the construction whose existence is asserted in Theorem \ref{mainthm}. By specifying some coarse information about the initial metrics $g_k(0)$, but before defining the flows $g_k(t)$ fully, we will be able to extract a limit Ricci flow $g_\infty(t)$ already in 
Section \ref{get_a_limit_sect}.
In Section \ref{gkdef_mainsect}, we work towards defining the flows $g_k(t)$ fully by constructing various constraints and barriers, finally being able to give the precise definition
of $g_k(t)$ in Section \ref{gkdef_sect}.
In Section \ref{apriori_est_sect} we derive the necessary \emph{a priori} estimates on the flows $g_k(t)$ and their limit $g_\infty(t)$ in order to be able to prove Theorem \ref{mainthm}.
Those unfamiliar with Ricci flow in two dimensions should begin with Appendix \ref{2Dappendix}.

\subsection{Other compactness assertions, true and false}
\label{trueorfalse}

There are a number of compactness assertions along the lines of Theorem \ref{ct2} which can be imagined or found in the literature, some of which are true, and some of which are false. 

The following simple variant will work in many applications. See 
\cite[Appendix E]{KLv4} for an essentially equivalent result.
\begin{thm}[Compactness of Ricci flows: Extension 2]
\label{ct3}
In the situation of Theorem \ref{ct2}, suppose we strengthen hypothesis (i) to the hypothesis that:
\begin{compactenum}[(i)$'$]
\item
for all $r>0$, there exist $K\in\N$ dependent on $r$, and $M<\infty$ {\bf independent} of $r$, such that for all $t\in (a,b)$ and for all $i\geq K$, there holds
$$\sup_{B_{g_i(0)(p_i,r)}} 
\big|\Rm(g_i(t))\big|_{g_i(t)} \leq M.$$
\end{compactenum}
Then the resulting limit flow $(\m,g(t))$ is complete
for all $t\in (a,b)$.
\end{thm}

\begin{proof}
For an arbitrary $r>0$, we will aim to control the curvature of $g(t)$ on $B_{g(0)}(p,r/2)\subset\subset\m$.
(Note that $(\m,g(0))$ is complete.)
By hypothesis, for sufficiently large $i$, we have 
$$\sup_{B_{g_i(0)(p_i,r)}} 
\big|\Rm(g_i(t))\big|_{g_i(t)} \leq M$$
for all $t\in (a,b)$, and thus by definition of the convergence
\eqref{cgnce_statement}, we must have
$$\sup_{B_{g(0)(p,r/2)}} 
\big|\Rm(g(t))\big|_{g(t)} \leq M.$$
Because $r>0$ was arbitrary, we deduce that $g(t)$ is a \emph{bounded curvature} Ricci flow which is complete at $t=0$, and is therefore complete for all $t\in (a,b)$.
\end{proof}

Implicit above is the well-known idea that control on the curvature leads to control on the evolution of distances
(e.g. \cite[Lemma 5.3.2]{RFnotes}).
The presence of curvature bounds thus often allows one to equivalently consider balls defined with respect to the metric at any time. However, the next result will indicate that this can sometimes not be the case.

\begin{thm}[Compactness of Ricci flows: Extension 3]
\label{ct4}
In the situation of Theorem \ref{ct2}, suppose we strengthen hypothesis (i) to the hypothesis that:
\begin{compactenum}[(i)$''$]
\item
for all $r>0$, there exists $M=M(r)<\infty$ such that for all $t_0, t\in (a,b)$ and for all $i$, there holds
$$\sup_{B_{g_i(t_0)(p_i,r)}} 
\big|\Rm(g_i(t))\big|_{g_i(t)} \leq M.$$
\end{compactenum}
Then the resulting limit flow $(\m,g(t))$ is complete
for all $t\in (a,b)$.
\end{thm}

\begin{proof}
We will see that
for given $t_0\in (a,b)$, the revised hypothesis $(i)''$ will ensure the completeness of $(\m,g(t_0))$.
It suffices to show that for an arbitrary $r_0>0$, 
the ball $B_{g(t_0)}(p,r_0)$ is compactly contained in $\m$.
Fix an arbitrary point $x\in B_{g(t_0)}(p,r_0)$, and a
smooth curve $\ga:[0,1]\to\m$ from $p$ to $x$, of length less than $2r_0$ with respect to $g(t_0)$.
By definition of the convergence \eqref{cgnce_statement} and 
hypothesis $(ii)''$, 
for sufficiently large $i$ we will have, say,
$$\big|\Rm(g(t))\big|_{g(t)}(y) \leq M(3r_0)+1=:\bar M$$ 
for all $y$ in the image of $\ga$ and all $t\in (a,b)$.
%
Then the length $\cl(\ga)$ of $\ga$ with respect to the evolving metric is controlled by
$$\cl_{g(0)}(\ga)\leq e^{C|t_0|} \cl_{g(t_0)}(\ga)
\leq e^{C|t_0|}2r_0,$$
for some $C<\infty$ depending only on $\bar M$
(see for example \cite[Lemma 5.3.2]{RFnotes}).
In particular, we find that $x\in B_{g(0)}(p,e^{C|t_0|}2r_0)$,
and because $x\in B_{g(t_0)}(p,r_0)$ was arbitrary, we conclude that
$$B_{g(t_0)}(p,r_0)\subset B_{g(0)}(p,e^{C|t_0|}2r_0)
\subset\subset\m,$$
as required.
\end{proof}

Theorem \ref{ct4} has sometimes been asserted with $t_0=t$, but 
when the upper bound $M$ is allowed to depend on $r$, this gives a weaker hypothesis that is not strong enough to establish the existence of the limit Ricci flow.

\emph{Acknowledgements:} Thanks to Esther Cabezas-Rivas for useful conversations, and to Oliver Schn\"urer for hosting us at Konstanz where these conversations took place. Thanks also to Gregor Giesen for useful comments. Supported by The Leverhulme Trust.

\section{Construction of the example}
\label{main_construct_sect}

In this section we prove Theorem \ref{mainthm}.
Throughout the construction, $g_{hyp}$ will denote the complete hyperbolic metric on $\Do$.
All metrics and flows are assumed to respect the underlying conformal structure of $D$ and its subsets (see Appendix \ref{2Dappendix}).

\subsection{Extraction of a limit}
\label{get_a_limit_sect}

We will define the Ricci flows $g_k(t)$ on $D$ by specifying their initial metrics $g_k(0)$ and then appealing to the global existence result of \cite[Theorem 1.3]{GT2}. Away from the origin, the initial metrics will coincide with the complete hyperbolic metric $g_{hyp}$ on $\Do$, for sufficiently large $k$. Moreover, writing $g_{poin}$ for the Poincar\'e metric (i.e. the complete hyperbolic metric on $D$) we will ensure that $g_k(0)\geq g_{poin}$ for each $k$. These two properties are enough to show convergence of the subsequent flows $g_k(t)$ to some limit $g_\infty(t)$ in the sense of Part 1 of Theorem \ref{mainthm}, after passing to a subsequence, as we now show.

\begin{lemma}
\label{get_a_limit}
Suppose $g_k(0)$ is any sequence of smooth metrics on $D$ (respecting the conformal structure of $D$) with the property that for all $\ep\in (0,1)$, and for sufficiently large $k$ (depending on $\ep$) we have $g_k(0)=g_{hyp}$ on $D\backslash D_\ep$. Suppose further that $g_k(0)\geq g_{poin}$ throughout $D$ for each $k$, and define $g_k(t)$ to be the subsequent Ricci flows for $t\in [0,\infty)$ which are given by \cite[Theorem 1.3]{GT2}. Then there exists a smooth Ricci flow $g_\infty(t)$ on $\Do$
for $t\in [0,\infty)$
with $g_\infty(0)=g_{hyp}$ such that after passing to a subsequence, we have
$$g_k(t)\to g_\infty(t)$$
smoothly locally on $(\Do)\times [0,\infty)$ as $k\to\infty$.
For all $t\in[0,\infty)$, we have 
\begin{equation}
\label{hyp_control}
g_\infty(t)\leq (1+2t)g_{hyp}.
\end{equation}
Moreover, if there exists $a>0$ such that $g_\infty(t)$ is complete for all $t\in [0,a]$, then we have equality in \eqref{hyp_control} for all $t\in [0,a]$.
%
\end{lemma}

Note that this lemma gives no information about what is happening to the flow $g_\infty(t)$ near the origin, except on a time interval where we assume completeness.
We will need to make further assumptions on $g_k(0)$, and derive a number of estimates, to get this finer control.

A basic tool in the proof will be the following:

\begin{lemma}
\label{hyp_comparison}
Let $(\m,H)$ be a complete hyperbolic surface.
If $g(t)$ is any conformally equivalent Ricci flow on $\m$ for $t\in [0,T]$ such that $g(0)\leq H$ then for all $t\in [0,T]$
\begin{equation}
\label{below_hyp}
g(t)\leq(1+2t)H.
\end{equation}
If in addition $g(0)=H$ and $g(t)$ is complete for all $t\in [0,T]$, then we have equality in \eqref{below_hyp}.
\end{lemma}

In the context of the last part of this lemma, note that 
it is a conjecture (see \cite{JEMS} and \cite[Conjecture 1.5]{GT2}) that instantaneously complete Ricci flows on surfaces are determined solely by their initial metric. 

\begin{proof} (Lemma \ref{hyp_comparison}.)
Because $(1+2t)H$ is the maximally stretched Ricci flow on 
$\m$ starting at $t=0$ with $H$ 
(see \cite[Theorem 1.3 and Theorem 1.8]{GT2})
we automatically have
\begin{equation}
\label{guppercontrol}
(1+2t)H\geq g(t)
\end{equation}
on $\m$, for all $t\in [0,T]$.
Assume from now that $g(t)$ is complete for all $t\in [0,T]$, and that $g(0)=H$. Chen's \emph{a priori} estimate for the curvature of a complete Ricci flow \cite[Corollary 2.3(i)]{chen} tells us that 
$$K[g(t)]\geq -\frac{1}{2t+1}$$ 
for all $t\in [0,T]$. Yau's Schwarz lemma (see the discussion in \cite[Section 2]{GT1}) then tells us that 
$$g(t)\geq (1+2t)H.$$
\end{proof}

\begin{proof} (Lemma \ref{get_a_limit}.)
Because the flows $g_k(t)$ arise from \cite[Theorem 1.3]{GT2},
they are \emph{maximally stretched}, and in particular,
because we know that $g_k(t)$ starts above the metric $g_{poin}$:
$$g_k(0)\geq g_{poin},$$
it must continue to lie above any Ricci flow starting at $g_{poin}$ 
and in particular,
$$g_k(t)\geq (1+2t)g_{poin}.$$
Thus, writing $g_k(t)=e^{2v_k}|\dz|^2$ (see Appendix \ref{2Dappendix}) we obtain a uniform lower bound for $v_k$ over the entire space-time region $D\times [0,\infty)$. Indeed, because $g_{poin}=(\frac{2}{1-|z|^2})^2|\dz|^2$, we have $v_k >0$ throughout.

To be able to extract a limit $g_\infty(t)$, it remains to establish an \emph{upper} bound for $v_k$ over compact subdomains of $(\Do)\times [0,\infty)$ that is independent of $k$ (but may depend on the subdomain) because then parabolic regularity theory will apply, yielding uniform $C^l$ bounds, and hence the desired compactness.

To obtain an upper bound over some region $D_{1-\ep}\backslash D_{2\ep}\times [0,T]$, for $\ep\in (0,1/3)$ and $T>0$, first consider the complete hyperbolic metric $g_\ep$ on $D\backslash D_{\ep}$. By the maximum principle (or direct computation as in Section \ref{insulate_sect} or Appendix \ref{2Dappendix}) 
we have the inequality
$$g_\ep\geq g_{hyp} = g_k(0)$$
over $D\backslash D_{\ep}$ for sufficiently large $k$.
By Lemma \ref{hyp_comparison} applied to the hyperbolic surface 
$(D\backslash D_\ep,g_\ep)$ and the Ricci flow $g_k(t)$,
we then have
\begin{equation}
\label{gkuppercontrol}
g_k(t)\leq(1+2t)g_\ep\leq(1+2T)g_\ep
\end{equation}
on $D\backslash D_{\ep}$, for all $t\in [0,T]$.
Restricting to $D_{1-\ep}\backslash D_{2\ep}$
then gives the required upper bound for $v_k$, and allows the limit $g_\infty(t)$ to be extracted.
By construction, we have $g_\infty(0)=g_{hyp}$, so another application of Lemma \ref{hyp_comparison} gives 
\eqref{hyp_control}, together with equality over any initial time interval $[0,a]$ on which $g_\infty(t)$ is complete.
\end{proof}

\subsection{Definition of the approximating Ricci flows $g_k(t)$}
\label{gkdef_mainsect}

In Section \ref{get_a_limit_sect} we indicated part of the definition of the Ricci flows $g_k(t)$ by giving their initial metrics away from the origin, and this was enough to be able to extract a limit and obtain some coarse estimates. However, the precise behaviour of the flows $g_k(t)$ depends very much on the definition of $g_k(0)$ near the origin.

We know that each of the metrics $g_k(0)$ will be the hyperbolic metric $g_{hyp}$ away from the origin. But near the origin we will make them a carefully sized cigar metric.
The impatient reader could jump forwards to Section \ref{gkdef_sect} to get a definition for $g_k(0)$, but we will carefully motivate this construction first.

\subsubsection{Insulating bounds}
\label{insulate_sect}

It is an important principle in Ricci flow that local regions cannot be too badly influenced by remote regions in too short a time \cite[\S 10]{P1}. Here we quantify some of the barrier arguments from Section \ref{get_a_limit_sect} to give some elementary local upper bounds on our Ricci flows $g_k(t)$ before we have even given their definition near the origin.

These estimates will effectively allow us to argue that in certain situations, two separate parts of a Ricci flow can evolve more or less independently, insulated from the effects of the other.

Recall that in Section \ref{main_construct_sect}, we always denote the hyperbolic metric on $\Do$ by $g_{hyp}$.

\begin{lemma}
\label{insulate_new}
Suppose $k\geq 1$ and $g(t)$ is any Ricci flow on $D$ (respecting the conformal structure of $D$) with the property that $g(0)\leq g_{hyp}$ on $D\backslash D_{e^{-2k}}$.
Then for all $t\in [0,\half]$, on the smaller annulus
$D\backslash D_{e^{-k}}$ we have
$$g(t)\leq \frac{\pi^2}{2} g_{hyp}.$$
Equivalently, working with respect to logarithmic conformal coordinates $s$ and $\th$ (see Appendix \ref{2Dappendix}) and writing the Ricci flow as $g(t)=e^{2u(s,\th,t)}(ds^2+d\th^2)$, we claim that 
if $u(s,\th,0)\leq -\ln s$ for $s\in (0,2k]$ and all $\th$, then
for all $t\in [0,\half]$ and $s\in (0,k]$, and all $\th$, we have
$$u(s,\th,t)\leq \half\ln \frac{\pi^2}{2} - \ln s.$$
\end{lemma}
We stress that even if we assumed that $g(0)=g_{hyp}$ on 
$D\backslash D_{e^{-2k}}$,
no lower bound of this type could be proved in this generality.
It is convenient to weaken this lemma into the exact form we will require in applications to $g_k(t)$. In practice, the Ricci flow $g_k(t)$ will satisfy the hypotheses with the same $k$.

\begin{cor}
\label{insulate3}
Suppose $k\geq 1$ and $g(t)=e^{2u(s,t)}(ds^2+d\th^2)$ is any rotationally-symmetric Ricci flow on $D$ with the property that $u(s,0)=-\ln s$ for $s\in (0,2k]$.
Then 
\begin{enumerate}
\item
for all $t\in [0,\half]$, we have
$u(k,t)\leq \half\ln 5 - \ln k$;
\item
for $s\in (0,k]$ we have $u(s,\half)\leq \half\ln 10 -\ln s$.
\end{enumerate}
\end{cor}

\begin{proof} (Lemma \ref{insulate_new}.)
The statement of the lemma may give the misleading impression that we are able to use some time-dependent dilation of $g_{hyp}$ directly as a barrier, but this is not possible with no information about $g(0)$ on $D_{e^{-2k}}$.
Instead we consider the complete hyperbolic metric $g_{e^{-2k}}$ on $D\backslash D_{e^{-2k}}$, which can be seen to satisfy 
$g(0)\leq g_{hyp}\leq g_{e^{-2k}}$ where it is defined, as in Section \ref{get_a_limit_sect}.
By the first part of Lemma \ref{hyp_comparison}, we have 
\begin{equation}
\label{hidden_est}
g(t)\leq (1+2t)g_{e^{-2k}}
\end{equation}
for all $t\in [0,\half]$.
We need the exact expression 
$$g_{e^{-2k}}=\left[
\frac{\pi}{2k\sin(\frac{s\pi}{2k})}
\right]^2(ds^2+d\th^2)$$
which allows us to translate \eqref{hidden_est} to
\begin{equation}
\label{uupper}
u(s,t)\leq 
-\ln \left[\frac{2k\sin(\frac{s\pi}{2k})}{\pi}\right]
+\half\ln (1+2t).
\end{equation}
Using the fact that $\sin(\frac{\pi x}{2})\geq x$ for $x\in [0,1]$, we find that for $s\in (0,k]$ and $t\in [0,\half]$, we have
$$u(s,t)\leq 
-\ln \left[\frac{2s}{\pi}\right]
+\half\ln (1+2t)
\leq -\ln s + \half\ln \left[\frac{\pi^2}{2}\right].$$
\end{proof}

\subsubsection{The cigar}

It will be most convenient to take logarithmic polar coordinates 
$s$ and $\th$ as defined in Appendix \ref{2Dappendix}
(i.e. so that $z=x+iy=e^{-(s+i\th)}$, where $x$ and $y$ are the standard coordinates on the plane).
In these coordinates, Hamilton's \emph{cigar} Ricci soliton metric can be defined to be $e^{2\cc(s)}(ds^2+d\th^2)$, for $s\in\R$,
where
$$\cc(s):=-\half\ln\left[1+e^{2s}\right],$$
and this induces the corresponding Ricci (soliton) flow $$(s,t)\mapsto\cc(s+2t).$$

In practice, we will need scaled (and translated) forms of the cigar, so for $\la>0$, we define
$$\cc_\la(s):=-\half\ln\la+\cc(s),$$ 
which corresponds to scaling the curvature by a factor $\la$.
The subsequent Ricci (soliton) flow is
$$(s,t)\mapsto
\cc_\la(s+2\la t).$$ 

\begin{rmk}
\label{cigar_rmk}
The conformal factor $\cc(s)$ has a number of obvious properties that we shall require, namely:
\begin{enumerate}
\item
$\sup_{s\in\R}\cc(s)=0$; 
\item
$\cc(s)$ is a decreasing function of $s$;
\item
$\cc(s)\leq -s$ for all $s\in \R$;
\item
$\cc(0)=-\half\ln 2$;
\item
$\cc(s)\geq -\half\ln 2 -s$ for all $s\geq 0$.
\end{enumerate}
\end{rmk}

\subsubsection{Cigar barriers}

The metric $g_{hyp}$ can be written with respect to $(s,\th)$ coordinates as $g_{hyp}=e^{2u_{hyp}(s)}(ds^2+d\th^2)$, for $s>0$, where 
$$u_{hyp}(s)=-\ln s.$$
We will construct the rotationally-symmetric metric $g_k(0)$ that will induce the Ricci flow $g_k(t)=:e^{2u_k(s,t)}(ds^2+d\th^2)$ (for $s>0$) by setting 
$u_k(s,0)=u_{hyp}(s)$ for $s\leq 2k$, and then extending $u_k(s,0)$ to $s>2k$ as an appropriate function which 
for large enough $s$ will coincide with the conformal factor of an appropriately scaled and translated cigar.

In fact, for each $k$ we will be constructing two cigar Ricci soliton flows $\cc_k^{lower}(s,t)$ and $\cc_k^{upper}(s,t)$ which will be lower and upper barriers respectively for $u_k(s,t)$ over the range $(s,t)\in [k,\infty)\times [0,\half]$, and so that $u_k(s,0)$ will coincide with $\cc_k^{lower}(s,0)$ for large enough $s$.
The idea of using upper and lower cigar barriers in this way was introduced in \cite{GT3}.

To decide \emph{which} cigars to take, we start with $\cc_k^{upper}$, and ask that at time $t=\half$, it is `centred at $k$', by which we mean that $\cc_k^{upper}(s,\half)=\cc_\la(s-k)$ for some $\la>0$. We want to choose $\la$ just small enough so that 
$\cc_k^{upper}(k,t)$ really must lie above $u_k(k,t)$
for $t\in [0,\half]$, and to this end, we set
$\la=\frac{k^2}{10}$ to force
\begin{equation}
\label{upper_def}
\begin{aligned}
\cc_k^{upper}(s,t)&=\cc_{\frac{k^2}{10}}\left(s-k+\frac{2k^2}{10}\left(t-\half\right)\right)\\
&=\cc_{\frac{k^2}{10}}\left(s-k-\frac{k^2}{10}+\frac{2k^2}{10}t\right).
\end{aligned}
\end{equation}
Indeed, note that by Part 1 of Corollary \ref{insulate3} and Parts 2 and 4 of Remark \ref{cigar_rmk} we have
\begin{equation}
u_k(k,t)\leq \half\ln 5 - \ln k
= \cc_{\frac{k^2}{10}}(0)
= \cc_k^{upper}(k,\half)
\leq \cc_k^{upper}(k,t),
\end{equation}
for $t\in [0,\half]$. Then as long as we choose $u_k(s,0)$ so that
$$u_k(s,0)\leq \cc_k^{upper}(s,0) \qquad\text{ for }s\geq k,$$
we will be able to apply the comparison principle for $s\geq k$ (strictly speaking we should return to working on the disc 
$\overline{D_{e^{-k}}}$ 
and apply it there, as in \cite{GT3}) to deduce that
\begin{equation}
\label{upper_cigar_est}
u_k(s,t)\leq \cc_k^{upper}(s,t),
\end{equation}
for $s\geq k$ and $t\in [0,\half]$.

We now turn to $\cc_k^{lower}$. We ask that at $t=0$, this lower barrier be translated by the same amount as the upper barrier $\cc_k^{upper}$, and this forces us to take 
$\cc_k^{lower}(s,t)=\cc_\la(s-k-\frac{k^2}{10}+2\la t)$, for some $\la>0$.
We choose $\la$ large enough so that 
$\sup (\cc_k^{lower})=-\half\ln\la$ is a little lower than 
$u_k(2k,0)=u_{hyp}(2k)=-\ln(2k)=-\half\ln(4k^2)$,
and we can achieve this by taking $\la=5k^2$, for example, giving
\begin{equation}
\label{low_enough}
\sup_{s\in\R} \,\cc_k^{lower}(s,0)\leq u_k(2k,0)=u_{hyp}(2k).
\end{equation}
This then determines 
\begin{equation}
\label{lower_def}
\cc_k^{lower}(s,t)=\cc_{5k^2}\left(s-k-\frac{k^2}{10}+10k^2 t\right).
\end{equation}

\subsubsection{Definition of $g_k(t)$}
\label{gkdef_sect}

Now that we are equipped with our cigar barriers $\cc_k^{upper}$ and $\cc_k^{lower}$ from \eqref{upper_def} and \eqref{lower_def}, we are in a position to define rotationally-symmetric flows
$g_k(t)=:e^{2u_k(s,t)}(ds^2+d\th^2)$ (for $s>0$) by determining the initial data $u_k(\cdot,0)$.
We choose any smooth $u_k(\cdot,0)$ satisfying:
\begin{enumerate}
\item
$u_k(s,0)=u_{hyp}(s)$ for $s\in (0,2k]$;
\item
$u_k(s,0)\leq\cc_k^{upper}(s,0)$ for all $s\geq 2k$;
\item
$u_k(s,0)\geq\cc_k^{lower}(s,0)$ for $s\geq 2k$.
\item
$u_k(s,0)=\cc_k^{lower}(s,0)$ for $s\geq 3k$.
\end{enumerate}
Clearly this can be done in a completely systematic way if desired. As we have mentioned already, the flows $g_k(t)$ are taken to be the maximally stretched Ricci flows constructed in 
\cite[Theorem 1.3]{GT2} corresponding to the initial metrics $g_k(0)$. Note that by the uniqueness of maximally-stretched solutions \cite[Theorem 1.3]{GT2}, the flows must retain the rotational symmetry of the initial metrics.
Strictly speaking, we will pass to a subsequence to obtain the Ricci flows $g_k(t)$ of Theorem \ref{mainthm}.

\subsection{\emph{A priori} estimates for $g_k(t)$ and $g_\infty(t)$, and a proof of Theorem \ref{mainthm}}
\label{apriori_est_sect}
It is the sequence $g_k(t)$ from Section \ref{gkdef_sect} that we put into Lemma \ref{get_a_limit} to get a limit $g_\infty(t)$. We should pause to verify the hypothesis that $g_k(0)\geq g_{poin}$, but this is easy to see by inspection for large enough $k$. Indeed for $s\in (0,2k]$, the maximum principle (or direct computation -- see Appendix \ref{2Dappendix}) is telling us that
$$g_k(0)=g_{hyp}\geq g_{poin},$$
while for $s\geq 2k$ and $k\geq 10$, the conformal factor of $g_k(0)$ is coarsely estimated (using Part 3 of the definition of $g_k(0)$ and Parts 2 and 5 of Remark \ref{cigar_rmk}) by
\begin{equation}
\begin{aligned}
u_k(s,0)&\geq \cc_k^{lower}(s,0)
=-\half\ln (5k^2) +\cc(s-k-\frac{k^2}{10})\\
&\geq-\half\ln (5k^2) +\cc(s-2k)
\geq -\half\ln (5k^2) -\half\ln 2 - (s-2k)\\
&\geq -s + \ln 2 \geq -\ln\sinh s,
\end{aligned}
\end{equation}
and $-\ln\sinh s$ is the conformal factor of $g_{poin}$
(see Appendix \ref{2Dappendix}).

We are thus truly in a situation where we can apply Lemma \ref{get_a_limit} to extract a limit: There exists a Ricci flow
$g_\infty(t)$ on $\Do$ for $t\in [0,\infty)$
with $g_\infty(0)=g_{hyp}$ such that after passing to a subsequence, we have
$$g_k(t)\to g_\infty(t)$$
smoothly locally on $(\Do)\times [0,\infty)$ as $k\to\infty$.
This is to be the sequence and limit of Theorem \ref{mainthm}, and we have already proved Part 1.

To prove Part 2 of Theorem \ref{mainthm}, note that Lemma \ref{get_a_limit} also tells us that if there existed $a>0$ such that $g_\infty(t)$ were complete for all $t\in [0,a]$, then we would have 
\begin{equation}
g_\infty(t)= (1+2t)g_{hyp},
\end{equation}
for all $t\in [0,a]$, as desired. This completeness will be a consequence of:

\begin{lemma}
\label{shorttimecompleteness}
Given $g_\infty(t)$ constructed as above as a subsequential limit of Ricci flows
$g_k(t)$ as defined in Section \ref{gkdef_sect},
we have
$$g_\infty(t)\geq \frac{1}{10}g_{hyp}$$
for all $t\in [0,\frac{1}{100}]$.
\end{lemma}
Thus, modulo a proof of this lemma, we can take $a=\eta=\frac{1}{100}$ to establish Part 2 of Theorem \ref{mainthm}.

\begin{proof} (Lemma \ref{shorttimecompleteness}.)
The main ingredient is the lower cigar Ricci soliton 
defined by \eqref{lower_def}. By Parts 1 and 3 of the definition of $g_k(0)$ from Section \ref{gkdef_sect}
(keeping in mind \eqref{low_enough}) we know that
the conformal factor $u_k(s,t)$ of $g_k(t)$ satisfies
\begin{equation}
\label{below_initially}
u_k(s,0)\geq \cc_k^{lower}(s,0)\qquad\text{ for all }s>0.
\end{equation}
The Ricci flows $g_k(t)$ were constructed using \cite[Theorem 1.3]{GT2}, and thus they are maximally stretched. 
But $\cc_k^{lower}(s,t)$ induces a Ricci flow on the plane which can be restricted to $D$, and which initially lies below $g_k(0)$ by \eqref{below_initially}. By the definition of maximally stretched, or by Lemma \ref{thm:direct-comp-principle}, it must continue to lie below, that is
$$u_k(s,t)\geq \cc_k^{lower}(s,t)\qquad\text{ for all }s>0\text{ and }
t>0.$$
In particular, for $t\in [0,\frac{1}{100}]$, we have
\begin{equation}
\label{edge_control}
\begin{aligned}
u_k(k,t)&\geq \cc_k^{lower}(k,t)\geq \cc_k^{lower}(k,\frac{1}{100})
=\cc_{5k^2}(0)=-\half\ln (5k^2)-\half\ln 2\\
&=-\ln k -\half\ln10.
\end{aligned}
\end{equation}
Given that $u_k(s,0)=-\ln s$ for $s\in (0,k]$, we can now conclude the proof if we can use $s\mapsto -\ln s-\half\ln10$ as a lower barrier for $u_k(s,t)$ over the time interval $t\in [0,\frac{1}{100}]$, which would imply $g_k(t)\geq \frac{1}{10}g_{hyp}$ on $D\backslash D_{e^{-k}}$ and hence $g_\infty(t)\geq \frac{1}{10}g_{hyp}$ on $\Do$ for all $t\in [0,\frac{1}{100}]$, as desired.

To verify that $s\mapsto -\ln s-\half\ln10$ is a lower barrier for $u_k(s,t)$ over the time interval $t\in [0,\frac{1}{100}]$,
we first note that just as in the proof of Lemma \ref{get_a_limit},
we have $g_k(t)\geq (1+2t)g_{poin}\geq g_{poin}$ for all $t\geq 0$, and thus $u_k(s,t)\geq -\ln\sinh s$, which ensures that for each $t\geq 0$, the function $u_k(s,t)$ converges to infinity at a $k$-independent rate as $s\downto 0$. We then consider, 
for each $\de>0$, the function $l:(-\de,k]\to\R$ defined by 
$$l(s)=-\ln (s+\de) -\half\ln 10,$$ 
which satisfies $e^{-2l}l_{ss}=10>0$
(i.e. $l$ is a subsolution of the Ricci flow). For $s\in (0,k]$, we have
$$l(s)< -\ln s-\half\ln 10\leq u_k(s,0),$$
and for $t\in [0,\frac{1}{100}]$, we have
$$l(k)< -\ln k-\half\ln 10\leq u_k(k,t),$$
by \eqref{edge_control}.
Then because $u_k$ solves $\pl{u}{t}=e^{-2u}u_{ss}$, the comparison principle forces it to remain above the function $l$, and we have for each $s\in (0,k]$ and $t\in [0,\frac{1}{100}]$ the required estimate
$$u_k(s,t)\geq l(s)\stackrel{\de\downto 0}\longrightarrow
-\ln s-\half\ln 10.$$
\end{proof}

Having established Part 2, 
the next step in the proof of Theorem \ref{mainthm} is to establish Parts 3 and 4, which assert that the limit $(\Do,g_\infty(t))$ is not complete for $t\geq 1$, and that $g_k(t)$ converges smoothly locally on $D\times [1,\infty)$ to an 
extension of $g_\infty(t)$ to the whole of of $D$ for $t\geq 1$.
For that, the main ingredients will be the upper cigar barrier $\cc_k^{upper}$ and a key estimate from \cite{revcusp}. The first step is:
\begin{lemma}
\label{thalfcontrol}
The Ricci flows $g_k(t)$ as defined in Section \ref{gkdef_sect} satisfy
$$g_k\left(\half\right)\leq 10 g_{hyp}$$
on $\Do$.
\end{lemma}
We stress that we do \emph{not} have $g_k(0)\leq M g_{hyp}$ on $\Do$ for any $k$-independent constant $M<\infty$. We do have to wait a little until we have a good upper bound by a controlled multiple of $g_{hyp}$.

\begin{proof}
First, by \eqref{upper_cigar_est} and Part 3 of 
Remark \ref{cigar_rmk}, we have 
\begin{equation}
\begin{aligned}
u\left(s,\half\right) &\leq \cc_k^{upper}\left(s,\half\right)=
\cc_{\frac{k^2}{10}}\left(s-k\right)
=-\half\ln\left(\frac{k^2}{10}\right)+\cc(s-k)\\
&\leq -\ln k+\half\ln 10-(s-k)\\
&\leq -\ln s +\half\ln 10
\end{aligned}
\end{equation}
for all $s\geq k$.
On the other hand, by Part 2 of Corollary \ref{insulate3},
we also have 
$$u\left(s,\half\right)\leq -\ln s+\half\ln 10$$
for $s\in (0,k]$.
Thus, we have established the required estimate for all $s>0$.
\end{proof}
From time $t=\half$ onwards, Lemma \ref{thalfcontrol} will allow us to apply the following lemma, which is a special case of \cite[Lemma 3.3]{revcusp}, and applies without any rotational symmetry hypothesis.
\begin{lemma}
\label{flowupperbdlemma}
If $g(t)=e^{2v(t)}|\dz|^2$ is any smooth Ricci flow on
$D$, for $t\in [0,1]$, with $g(0)\leq g_{hyp}$ 
then there exists $\be<\infty$ universal such that
\begin{equation}
\sup_{D_{\half}}v(\cdot,t)\leq \frac{\be}{t}
\end{equation}
for $t\in (0,1]$.
\end{lemma}
This will be the main ingredient required to uniformly control the conformal factors of the $g_k(t)$ at time $t=\frac{3}{4}$:
\begin{lemma}
\label{later_control_at_origin}
The Ricci flows $g_k(t)=:e^{2v_k(t)}|\dz|^2$ as defined in Section \ref{gkdef_sect} satisfy
$$\sup_{D_{1/4}}v_k(\cdot,t)\leq C+t$$
for $t\geq \frac{3}{4}$ and for some universal $C<\infty$.
\end{lemma}
No such uniform upper bound could possibly hold for any 
$t\in [0,\frac{1}{100}]$.
\begin{proof}
Let us fix $k\in\N$, and define a parabolically rescaled Ricci flow \cite[\S 1.2.3]{RFnotes} by
$$g(t):=\frac{1}{10}g_k\left(\half+10t\right)$$
for $t\geq 0$. If we write $g(t)=e^{2v}|\dz|^2$, then 
this makes $v(\cdot,t)=-\half\ln 10+v_k(\cdot,\half+10t)$.
By Lemma \ref{thalfcontrol} we have
$$g(0)\leq g_{hyp},$$
and hence by Lemma \ref{flowupperbdlemma}, we have
$$\sup_{D_{\half}}v\leq \frac{\be}{t}$$
for universal $\be<\infty$, and $t\in (0,1]$. Translating back from $v$ to $v_k$, we find that
$$v_k\left(\cdot,\frac{3}{4}\right)=v\left(\cdot,\frac{1}{40}\right)+\half\ln10\leq 40\be+\half\ln 10=:\gamma,$$
on $D_{\half}$. Now we take the complete, homothetically expanding Ricci flow $G(t)=e^{2w}|\dz|^2$ on $D_\half$, for $t\geq \frac{3}{4}$, defined by
$$w(z,t):=\ln\left[\frac{1}{1/4-|z|^2}\right]+\half\ln [\al+2t],$$
where the universal constant $\al\in [1,\infty)$ is chosen large enough so that
$w(z,t)>\gamma$ for all $t\geq \frac{3}{4}$ and $z\in D_\half$.
Now 
$$w\left(\cdot,\frac{3}{4}\right)>\gamma\geq v_k\left(\cdot,\frac{3}{4}\right)$$
so exploiting the fact that $G(t)$ is the unique maximally stretched Ricci flow \cite{GT2} (or using the comparison principle Lemma \ref{thm:direct-comp-principle}) we find that
$$w\left(\cdot,t\right)\geq v_k\left(\cdot,t\right)$$
on $D_\half$, for all $t\geq\frac{3}{4}$, and thus 
$$\sup_{D_{1/4}}v_k(\cdot,t)\leq \ln \frac{16}{3}+\half\ln[\al+2t],$$
for all $t\geq\frac{3}{4}$, which is a little stronger than claimed in the lemma.
\end{proof}

Now we have controlled the conformal factors $v_k$ near the origin for $t\geq\frac{3}{4}$, we see that 
just as in the proof of Lemma \ref{get_a_limit},
parabolic regularity theory gives us $C^l$ control on $v_k$ for 
$t\in [1,T]$ (for arbitrary $T>1$) within a fixed neighbourhood of the origin. Also as in the proof of Lemma \ref{get_a_limit},
this allows us to pass to a further subsequence and get smooth local convergence of $g_k(t)$ to a limit Ricci flow on the whole space-time region $D\times [1,\infty)$. We still call this extended limit $g_\infty(t)$. We have proved the following corollary, which implies Parts 3 and 4 of Theorem \ref{mainthm}.

\begin{cor}
The Ricci flow $g_\infty(t)$ on $\Do$, constructed as above as a subsequential limit of Ricci flows $g_k(t)$ as defined in Section \ref{gkdef_sect}, can be extended smoothly for $t\geq 1$ to a Ricci flow on the whole disc $D$, and the extended Ricci flow is the smooth local (subsequential) limit of the flows $g_k(t)$
on $D\times [1,\infty)$.
\end{cor}

Finally, we remark that the applications of parabolic regularity theory which gave the limit $g_\infty(t)$, and its extension across the origin for $t\geq 1$, imply immediately the curvature bounds of Part 5 of Theorem \ref{mainthm}, which completes the proof.

\appendix

\section{Ricci flow in two-dimensions}
\label{2Dappendix}
Hamilton's Ricci flow equation
\beq
\label{RFeq}
\pl{}{t}g(t)=-2\,\Ric[g(t)]
\eeq
for an evolving Riemannian metric $g(t)$ on a manifold $\m$ simplifies in the case that $\m$ is two-dimensional. The flow then  preserves the conformal class of the metric and can be written
$$\pl{}{t}g(t)=-2K g(t),$$
where $K$ is the Gauss curvature of $g(t)$.
In this case, we may take local isothermal coordinates
$x$ and $y$, and write the flow 
$g(t)=e^{2u}|\dz|^2:=e^{2u}(dx^2+dy^2)$
for some locally-defined scalar time-dependent function $u$ which will then satisfy the local equation
\beq
\label{2DRFeq}
\pl{u}{t}=e^{-2u}\lap u = -K.
\eeq
We sometimes abuse terminology by referring also to $u$ itself as a Ricci flow.
If we have upper and lower bounds for $u$ in some local region, then one can apply parabolic regularity theory (e.g. \cite{LSU}) to obtain $C^k$ estimates on the interior.

In this paper, we are either considering standard coordinates $x$ and $y$ on the unit disc in the plane centred at the origin (or other rotationally symmetric subsets of the plane) or else the conformally equivalent logarithmic polar coordinates $s$ and $\th$ on the plane punctured at the origin, defined by 
$z=x+iy=e^{-(s+i\th)}$
for $s\in (0,\infty)$ or $s\in \R$, and $\th\in \R/(2\pi\Z)$.

In these coordinates we can write the conformal factors of the complete hyperbolic metrics $e^{2u(s)}(ds^2+d\th^2)$ on 
$D$, $\Do$ and $D\backslash D_\ep$, for $\ep\in (0,1)$, as
\begin{align*}
(D,g_{poin}) &: &\  u(s)&=-\ln\sinh s &\ s>0\\
(\Do,g_{hyp}) &: &\  u(s)&=-\ln s &\ s>0\\
(D\backslash D_\ep,g_{\ep}) &: 
&\ u(s)&=-\ln \left[\frac{-\ln\ep}{\pi}\sin\left(\frac{s\pi}{-\ln\ep}\right)\right]
&\ s\in (0,-\ln\ep).
\end{align*}
We see that $g_{poin}\leq g_{hyp}\leq g_{\ep}$, although this 
can be seen without computation via the maximum principle because
$D\backslash D_\ep \subset \Do\subset D$.

\section{Comparison principle}

The following standard comparison principle is a corrected version of \cite[Theorem A.1]{GT1}. 

\begin{lemma}[Direct comparison principle]
  \label{thm:direct-comp-principle}
  
  Let $\Omega\subset\mathbb{R}^2$ be an open, bounded domain and for some
  $T>0$ let $u\in C^{1,2}\bigl(\Omega\times(0,T)\bigr)\cap
  C\bigl(\bar\Omega\times[0,T]\bigr)$ and $v\in
  C^{1,2}\bigl(\Omega\times(0,T)\bigr)\cap
  C\bigl(\Omega\times[0,T]\bigr)$ both be solutions of the Ricci flow equation \eqref{2DRFeq}
  for the conformal factor of the metric.
  Furthermore, suppose that there exists some function $V\in C(\Omega)$ with $V(x)\to\infty$ as $x\to\partial \Omega$ such that
   for each $t\in[0,T]$ we have
  $v(x,t)\ge V(x)$ for all $x\in\Omega$. 
  If $v(x,0)\ge u(x,0)$ for all $x\in\Omega$, then $v\ge u$ on
  $\Omega\times[0,T]$.
\end{lemma}

\begin{proof} (Adjustment of that in \cite{GT1}.)
  For every $\varepsilon>0$ consider
  \[ v_\varepsilon(x,t) := v\left(x,\frac1\varepsilon\ln(\varepsilon t+1)\right)
  +\frac12\ln(\varepsilon t+1)\qquad\text{for all }
  (x,t)\in\Omega\times[0,T], \]
  which is well-defined since $\frac1\varepsilon\ln(\varepsilon t+1)\le
  t$ for all $t\ge0$. Observe that $v_\varepsilon$ is a slight modification of
  $v$, with $v_\varepsilon(\cdot,0)=v(\cdot,0)$, and $v_\varepsilon$ converges pointwise
  to $v$ as $\varepsilon\to0$, but in contrast to $v$ it is a
  strict supersolution of
  \eqref{2DRFeq}:
  \begin{align}
    \label{eq:rf-cf-mod}
    \left(\pddt v_\varepsilon-\ee^{-2v_\varepsilon}\Delta
      v_\varepsilon\right)(x,t) 
    &= \frac1{\varepsilon t+1}\left(\pddt v-\ee^{-2v}\Delta
      v\right)\left(x,\frac1\varepsilon\ln(\varepsilon t+1)\right)
    + \frac\varepsilon{2(\varepsilon t+1)} \notag\\
    &= \frac\varepsilon{2(\varepsilon t+1)} >0 
    \qquad\qquad \text{for }
    (x,t)\in\Omega\times[0,T].
  \end{align}
  We are going to prove $(v_\varepsilon-u)\ge0$ on
  $\Om\times[0,T]$ and conclude the theorem's statement by letting
  $\varepsilon\to0$. 

  Since by hypothesis $u$ is continuous on $\bar\Omega\times [0,T]$ and (for each $t\in[0,T]$)
  $v_\varepsilon(\cdot,t)\geq V$, where $V$ blows up near the boundary $\partial \Om$, there must exist a subdomain $\Omega'\subset\subset\Omega$ so that 
\[ (v_\varepsilon-u)(x,t)\geq 1 \]
throughout $(\Omega\backslash\Omega') \times [0,T]$.
  Now assume that $(v_\varepsilon-u)$ becomes negative in $\Omega\times[0,T]$,
  and define the time $t_0$ at which $(v_\varepsilon-u)$
  first becomes negative by 
  \[ t_0 := \inf\Bigl\{t\in[0,T]:
  \min_{x\in\Omega}(v_\varepsilon-u)(x,t)<0 \Bigr\}\in[0,T). \]
Then there exists a minimum $x_0\in\Omega$ of $(v_\ep-u)(\cdot,t_0)$ where  
  \[ (v_\varepsilon-u)(x_0,t_0) = 0,\quad
  \Delta(v_\varepsilon-u)(x_0,t_0)\ge0\quad\text{ and }\quad
  \pddt(v_\varepsilon-u)(x_0,t_0)\le0.\]
  Subtracting the Ricci flow equation
  \eqref{2DRFeq} from \eqref{eq:rf-cf-mod} 
  at this point $(t_0,x_0)$, we find
  \begin{align*}
    0 &< 
    \left(\pddt v_\varepsilon - \ee^{-2v_\varepsilon}\Delta
      v_\varepsilon\right)(x_0,t_0)    
    -\left(\pddt u - \ee^{-2u}\Delta u\right)(x_0,t_0) \\
    &= \pddt(v_\varepsilon -u)(x_0,t_0) -
    \ee^{-2u(x_0,t_0)}\Delta(v_\varepsilon-u)(x_0,t_0) \le 0,
  \end{align*}
  which is a contradiction.
  Therefore $v_\varepsilon\ge u$ on $\Omega\times[0,T]$, and
  the corresponding statement for $v$ follows by letting $\varepsilon\to0$.
\end{proof}

{\sc mathematics institute, university of warwick, coventry, CV4 7AL,
uk}\\
\url{http://www.warwick.ac.uk/~maseq}

\begin{thebibliography}{99}
\bibitem{CWv4} \textsc{E. Cabezas-Rivas} and \textsc{B. Wilking},
\emph{How to produce a Ricci Flow via Cheeger-Gromoll exhaustion.\/} Version 4. \url{http://arxiv.org/abs/1107.0606v4}

\bibitem{chen} 
\textsc{B.-L. Chen}, \emph{Strong uniqueness of 
the Ricci flow.\/} 
J. Diff. Geom. 
{\bf 82} (2009) 363--382.


\bibitem{ham_dask} \textsc{P. Daskalopoulos} and \textsc{R. S. Hamilton}, \emph{Geometric estimates for the logarithmic fast diffusion equation.\/} Comm. Anal. Geom. 
{\bf 12} (2004) 143--164.

\bibitem{dask_sesum} \textsc{P. Daskalopoulos} and \textsc{N. Sesum}, \emph{Type II extinction profile of maximal solutions to the Ricci flow in $\R^2$.\/} {\bf 20} (2010) 565--591.

\bibitem{GT1} \textsc{G. Giesen} and \textsc{P. M. Topping},
\emph{Ricci flow of negatively curved incomplete surfaces.\/}
Calc. Var. and PDE, {\bf 38} (2010), 357--367.

\bibitem{GT2} 
\textsc{G. Giesen} and \textsc{P. M. Topping},
\emph{Existence of Ricci flows of incomplete surfaces.\/} 
Comm. Partial Differential Equations, {\bf 36} (2011) 1860--1880. 
\url{http://arxiv.org/abs/1007.3146}

\bibitem{GT3} 
\textsc{G. Giesen} and \textsc{P. M. Topping},
\emph{Ricci flows with unbounded curvature.\/} Preprint (2011).
\url{http://arxiv.org/abs/1106.2493}

\bibitem{GTinprep} 
\textsc{G. Giesen} and \textsc{P. M. Topping}, In preparation.


\bibitem{hamcptness} \textsc{R. S. Hamilton}, 
\emph{A compactness property for
solutions of the Ricci flow.\/} Amer. J. Math. {\bf 117} (1995) 545--572.

\bibitem{formations} \textsc{R. S. Hamilton}, 
\emph{The formation of singularities
in the Ricci flow.\/} Surveys in differential geometry, Vol. II
(Cambridge, MA, 1993) 7--136, Internat. Press, Cambridge, MA, 1995.

\bibitem{KLGT} \textsc{B. Kleiner} and \textsc{J. Lott},
\emph{Notes on Perelman's papers.\/} 
Geometry and Topology, {\bf 12} (2008) 2587--2855.

\bibitem{KLv4} \textsc{B. Kleiner} and \textsc{J. Lott},
\emph{Notes on Perelman's papers.\/} Version 4 (23 Aug. 2011)
\url{http://arxiv.org/abs/math/0605667v4}

\bibitem{LSU}
{\sc O. A. Lady\v{z}enskaja}, {\sc V. A. Solonnikov} and 
{\sc N. N. Ural'ceva}, {\em Linear and
  {Q}uasi-linear {E}quations of {P}arabolic {T}ype.\/} {\bf 23}
Translations of Mathematical Monographs.
American Mathematical Society, Providence, 1968.

\bibitem{P1} \textsc{G. Perelman} 
\emph{The entropy formula for the Ricci flow 
and its geometric applications.\/}
\url{http://arXiv.org/abs/math/0211159}v1  (2002).



\bibitem{RFnotes} \textsc{P. M. Topping}, \emph{Lectures on the Ricci flow.\/}
L.M.S. Lecture notes series {\bf 325} C.U.P. (2006)
\url{http://www.warwick.ac.uk/~maseq/RFnotes.html}

\bibitem{JEMS} \textsc{P. M. Topping}, \emph{Ricci flow compactness 
via pseudolocality, and flows with incomplete initial metrics.\/}
J. Eur. Math. Soc. (JEMS) {\bf 12} (2010) 1429--1451.

\bibitem{rick_survey} \textsc{P. M. Topping},
\emph{Reverse bubbling in geometric flows.\/}
Appeared in the volume in honour of Rick Schoen's 60th birthday
(2011) International Press.

\bibitem{revcusp} \textsc{P. M. Topping},
\emph{Uniqueness and nonuniqueness for Ricci flow on surfaces: Reverse cusp singularities.\/} I.M.R.N., published online 2011.

\bibitem{Yau73} \textsc{S.-T. Yau}, \emph{Remarks on
conformal transformations.\/} J. Differential Geometry
{\bf 8} (1973), 369--381.
\end{thebibliography}
\end{document}